\documentclass[11pt]{article}

\usepackage[margin=1.12in]{geometry}
\usepackage[T1]{fontenc}
\usepackage{lmodern}
\usepackage{amsmath,amssymb,amsthm,mathtools}
\usepackage{enumitem}
\usepackage{microtype}
\usepackage{hyperref}

\newenvironment{keywords}{%
	\par\medskip
	\noindent\textbf{Keywords.}\enspace\ignorespaces
}{%
	\par\medskip
}

\newcommand{\MSC}[2][]{%
	\par\smallskip
	\noindent\textbf{#1 Mathematics Subject Classification.}\enspace #2\par
}

\hypersetup{colorlinks=true,linkcolor=blue,citecolor=blue,urlcolor=blue}

\newtheorem{theorem}{Theorem}[section]
\newtheorem{lemma}[theorem]{Lemma}
\newtheorem{corollary}[theorem]{Corollary}
\newtheorem{proposition}[theorem]{Proposition}

\theoremstyle{remark}

\newcommand{\Tr}{\operatorname{Tr}}
\newcommand{\spec}{\operatorname{spec}}
\newcommand{\diag}{\operatorname{diag}}

\newcommand{\bbC}{\mathbb C}
\newcommand{\cR}{\mathcal R}
\newcommand{\eps}{\varepsilon}

\newcommand{\Mn}{M_n(\bbC)}
\newcommand{\Hn}{\mathbb H_n}
\newcommand{\Pn}{\mathbb P_n}
\newcommand{\Psn}{\overline{\mathbb P}_n}

\title{Sharp Hyperbolic Cutoffs and Dimension-Sharp Counterexamples for Reverse Araki-Type Inequalities}

\author{
	Trung Dung Vuong\\
	\small High School for the Gifted, Vietnam National University Ho Chi Minh City, Vietnam\\
	\small \texttt{vtdung@ptnk.edu.vn}
}

\begin{document}
\maketitle

\begin{abstract}
	We study reverse Araki-type trace inequalities and log-majorizations beyond
	the exponent $2$. For arbitrary nonnegative nondecreasing weights, we show that
	$s=2$ is the sharp dimension-free boundary: for every $s>2$, explicit
	one-parameter $3\times3$ positive definite examples violate the reverse
	Liu--Cheng trace inequality and the corresponding dual formulation of
	Shi--Wei--Wang, whereas the reverse inequality remains valid for every
	$s\geq1$ in dimension $2$.
	
	For power weights, a larger region survives and is bounded by a sharp
	hyperbola. In normalized variables, for $s>2$,
	\[
	A^{r+s}B^s\succ_{\log}A^r(A^{1/2}BA^{1/2})^s
	\]
	holds for all positive semidefinite matrices in every finite dimension if and
	only if $0\leq r\leq s/(s-2)$; beyond this range, even the associated trace
	inequality fails for $3\times3$ positive definite matrices. Equivalently, for
	$0<p\leq q$ and $q>2p$, the sharp condition is
	$0\leq r\leq pq/(q-2p)$. Combined with the known all-$r$ regime
	$p\leq q\leq2p$, this completes the reverse log-majorization phase diagram.
\end{abstract}


\section{Introduction}

For a fixed dimension $n$, let $\Mn$ be the algebra of complex $n\times n$
matrices, let $\Hn$ be the real vector space of Hermitian matrices, and let
$\Pn$ and $\Psn$ denote the cones of positive definite and positive
semidefinite matrices, respectively. For a matrix $X$ with nonnegative
eigenvalues, write
\[
\lambda_1(X)\geq\cdots\geq\lambda_n(X)\geq0.
\]
For two such matrices, $X\prec_{\log}Y$ means
\[
\prod_{j=1}^k\lambda_j(X)
\leq
\prod_{j=1}^k\lambda_j(Y),
\qquad 1\leq k<n,
\]
together with $\det X=\det Y$. We use $X\succ_{\log}Y$ for the reverse
relation. Since log-majorization implies weak majorization of the eigenvalue
vectors, it also implies the corresponding trace inequality. Consequently, a
trace counterexample is enough to rule out a proposed log-majorization.

Products such as $A^{r+s}B^s$ and
$A^r(A^{1/2}BA^{1/2})^s$ need not be Hermitian. In log-majorization
statements, they are interpreted through their spectra. Indeed, if $A>0$,
$B\geq0$, and $C=A^{1/2}BA^{1/2}$, then
\[
A^{r+s}B^s
\sim
A^{(r+s)/2}B^sA^{(r+s)/2},
\qquad
A^rC^s
\sim
A^{r/2}C^sA^{r/2},
\]
where the matrices on the right are positive semidefinite. Singular cases are
handled by regularization; this is carried out explicitly in the proof of
Theorem~\ref{thm:sharp-power-log}.

We shall also use the elementary fact that $0\leq S\leq T$ and $E\geq0$ imply
\[
\Tr SE\leq\Tr TE,
\]
because
\[
\Tr(T-S)E
=
\Tr\bigl(E^{1/2}(T-S)E^{1/2}\bigr)\geq0.
\]

For positive semidefinite matrices $A,B$, the Araki--Lieb--Thirring inequality
states that
\begin{equation}\label{eq:ALT-trace}
	\Tr(A^{1/2}BA^{1/2})^s\leq\Tr A^sB^s,
	\qquad s\geq1,
\end{equation}
with the reverse inequality for $0\leq s\leq1$; see \cite{Araki}. Its standard
log-majorization refinement is
\begin{equation}\label{eq:ALT-log}
	(A^{1/2}BA^{1/2})^s
	\prec_{\log}
	A^{s/2}B^sA^{s/2},
	\qquad s\geq1;
\end{equation}
see, for example, \cite{AndoHiai,Hiai2016}.

Liu and Cheng \cite{LiuCheng} proved the monotone-weight extension of
\eqref{eq:ALT-trace} in the sublinear range: if $f$ is nonnegative and
nondecreasing on an interval containing $\spec(A)$, then
\begin{equation}\label{eq:LC-sublinear}
	\Tr f(A)A^sB^s
	\leq
	\Tr f(A)(A^{1/2}BA^{1/2})^s,
	\qquad 0\leq s\leq1.
\end{equation}
They conjectured the reverse inequality
\begin{equation}\label{eq:LC-conj1}
	\Tr f(A)A^sB^s
	\geq
	\Tr f(A)(A^{1/2}BA^{1/2})^s,
	\qquad s\geq1,
\end{equation}
and the reverse GBLP-type log-majorization
\begin{equation}\label{eq:LC-conj2}
	A^{r+q}B^q
	\succ_{\log}
	A^r\bigl(A^{p/2}B^pA^{p/2}\bigr)^{q/p},
	\qquad 0<p\leq q,\quad r\geq0.
\end{equation}
The trace consequence of \eqref{eq:LC-conj2} with $p=1$ and $q=s$ yields the
power-weight case $f(t)=t^r$ of \eqref{eq:LC-conj1}; see also
\cite{BebianoLemosProvidencia,Furuta2012,LemosSoares,
	MatsumotoNakamotoFujii}.

Shi, Wei, and Wang \cite{ShiWeiWang} proved \eqref{eq:LC-conj1} for
$1\leq s\leq2$ and established \eqref{eq:LC-conj2} for
\[
p\leq q\leq2p,\qquad r\geq0.
\]
They also formulated the corresponding questions in the remaining range
$s>2$, or equivalently $q>2p$ for power weights.

The previously known and new parts are separated as follows. The direct
monotone-weight inequality for $0\leq s\leq1$ and the reverse inequality for
$1\leq s\leq2$ are background results of Liu--Cheng and Shi--Wei--Wang,
respectively. Our contribution in the monotone-weight setting is the failure
for $s>2$, together with the proof that dimension $3$ is minimal. At the
power/log-majorization level, the all-$r$ regime $p\leq q\leq2p$ and the
reverse GBLP-type regime $q>2p$, $0\leq r\leq p$, are background. We determine
the entire remaining region and construct counterexamples exactly beyond its
boundary.

The present paper resolves these questions and identifies two distinct
boundaries. For arbitrary monotone weights, $s=2$ is the sharp
dimension-free endpoint: for every $s>2$, we construct explicit
one-parameter $3\times3$ positive definite counterexamples to
\eqref{eq:LC-conj1}. The same construction disproves the dual $s>2$
formulation of Shi--Wei--Wang, whose hypothesis requires
$x\mapsto x^s g(x)$ to be nonnegative and nonincreasing. In contrast,
\eqref{eq:LC-conj1} remains valid for every $s\geq1$ in dimension $2$, so the
obstruction is dimension-sharp.

For power weights, the positive region extends beyond $s=2$ up to a sharp
hyperbola. More precisely, in the reverse Liu--Cheng range $0<p\leq q$, we
prove that
\begin{equation}\label{eq:intro-corrected-phase}
	A^{r+q}B^q
	\succ_{\log}
	A^r\bigl(A^{p/2}B^pA^{p/2}\bigr)^{q/p}
\end{equation}
holds for all positive semidefinite matrices in every finite dimension if and
only if either
\[
p\leq q\leq2p,
\]
or
\[
q>2p
\quad\text{and}\quad
0\leq r\leq\frac{pq}{q-2p}.
\]
The first regime is the all-$r$ theorem of Shi--Wei--Wang. When $q>2p$,
previous reverse GBLP-type results cover $0\leq r\leq p$, while our new
affirmative range is
\[
p<r\leq\frac{pq}{q-2p}.
\]
For $r>pq/(q-2p)$, even the corresponding trace inequality fails for
$3\times3$ positive definite matrices. Under the normalization $p=1$,
$q=s$, this boundary becomes $r=s/(s-2)$ for $s>2$.

The counterexamples are explicit rank-one blow-up constructions rather than
numerical examples; the balance of their asymptotic exponents produces the
hyperbolic cutoff. The positive power-weight result follows from Tanahashi's
negative-power form of the Furuta inequality \cite{Tanahashi}. A largest
eigenvalue comparison is then lifted to log-majorization through
antisymmetric tensor powers. The two-dimensional result is obtained from an
endpoint comparison specific to $2\times2$ positive semidefinite matrices.

Section~\ref{sec:counterexamples} constructs the counterexamples and proves the
sharp obstruction beyond the hyperbola. Section~\ref{sec:tanahashi}
establishes the positive power-weight range. Section~\ref{sec:monotone}
derives the monotone-weight phase transition and the two-dimensional theorem.
Section~\ref{sec:corrected} returns to the original $(p,q,r)$ variables.
Throughout, phase diagrams refer to $0<p\leq q$; the complementary classical
range $0<q\leq p$ is not re-proved here.

\section{Counterexamples beyond $s=2$ and beyond the hyperbola}\label{sec:counterexamples}

We first give structural counterexamples. These examples are not numerical
accidents. The failure is caused by a large one-dimensional direction of $B$ which
is invisible in a chosen diagonal entry of $B^s$, but becomes visible after the
sandwiching by $A^{1/2}$.
The counterexamples are driven by the following rank-one blow-up estimate.

\begin{lemma}\label{lem:rank-one-fixed}
	Let $T_0\geq0$, let $w\neq0$, and, for a scalar parameter $M>0$, put
	\[
	T_M=T_0+Mww^*.
	\]
	Set $d=\|w\|^2$. Let $x$ be a unit vector such that $x\perp w$, and suppose
	\[
	h:=\langle x,T_0w\rangle\neq0.
	\]
	Then, for every $s>2$,
	\begin{equation}\label{eq:rank-one-fixed-asymp}
		\langle x,T_M^s x\rangle
		=
		|h|^2d^{s-3}M^{s-2}(1+o(1)),
		\qquad M\to\infty.
	\end{equation}
\end{lemma}

\begin{proof}
	The matrix $ww^*$ has the simple nonzero eigenvalue $d$. Hence the largest
	eigenvalue $\mu_M$ of $T_M$ satisfies $\mu_M=dM+O(1)$, its normalized eigenvector
	$q_M$ satisfies $q_M\to w/\sqrt d$, and all remaining eigenvalues remain bounded.
	Since $x\perp w$, the eigenvalue equation gives
	\[
	\mu_M\langle x,q_M\rangle=\langle x,T_0q_M\rangle,
	\]
	and therefore
	\[
	M\langle x,q_M\rangle\longrightarrow \frac{h}{d^{3/2}}.
	\]
	The top eigenvalue contributes
	\[
	\mu_M^s|\langle x,q_M\rangle|^2
	=|h|^2d^{s-3}M^{s-2}(1+o(1)),
	\]
	whereas the remaining spectral contribution is $O(1)$. Since $s>2$, this proves
	\eqref{eq:rank-one-fixed-asymp}.
\end{proof}

We now disprove the monotone-weight conjecture for every $s>2$ by explicit
one-parameter positive definite examples in dimension $3$.

\begin{theorem}\label{thm:monotone-counter}
	For every $s>2$, there exist $A,B\in\mathbb P_3$ and a nonnegative
	nondecreasing continuous function $f$ on an interval containing $\spec(A)$ such
	that
	\begin{equation}\label{eq:monotone-failure}
		\Tr f(A)A^sB^s
		<
		\Tr f(A)(A^{1/2}BA^{1/2})^s.
	\end{equation}
	Consequently, \eqref{eq:LC-conj1} is false in general.
\end{theorem}

\begin{proof}
	Fix
	\[
	0<b<a<1,
	\qquad
	\tau>0,
	\qquad
	\eta>\tau^2,
	\]
	and set
	$
	A=\diag(1,a,b).
	$
	Let
	\[
	u=\frac{e_2+e_3}{\sqrt2},
	\qquad
	v=\frac{e_2-e_3}{\sqrt2},
	\]
	and let $U$ be the unitary with columns $e_1,u,v$. Define
	\[
	B_M
	=U
	\begin{pmatrix}
		1&\tau&0\\
		\tau&\eta&0\\
		0&0&M
	\end{pmatrix}
	U^*,
	\qquad M>0.
	\]
	Since $\eta>\tau^2$, we have $B_M>0$.
	
	Choose a nonnegative nondecreasing continuous function $f$ such that
	$f(b)=f(a)=0$ and $f(1)=1$; for instance, one may take
	$f(t)=\max\{0,(t-a)/(1-a)\}$ on $[b,1]$. Then $f(A)=P$, where $P=e_1e_1^*$.
	In the basis $\{e_1,u,v\}$, $B_M$ is the direct sum of
	\[
	K=\begin{pmatrix}1&\tau\\ \tau&\eta\end{pmatrix}
	\]
	and the scalar $M$. Hence
	\begin{equation}\label{eq:monotone-left-bounded}
		\Tr PA^sB_M^s
		=\langle e_1,B_M^s e_1\rangle
		=\langle e_1,K^s e_1\rangle,
	\end{equation}
	which is independent of $M$.
	
	In the same basis,
	\[
	U^*A^{1/2}U
	=
	\begin{pmatrix}
		1&0&0\\
		0&\alpha&\beta\\
		0&\beta&\alpha
	\end{pmatrix},
	\qquad
	\alpha=\frac{\sqrt a+\sqrt b}{2},
	\quad
	\beta=\frac{\sqrt a-\sqrt b}{2}.
	\]
	Since $a\neq b$, $\beta\neq0$. Therefore
	\[
	U^*A^{1/2}B_MA^{1/2}U=C_0+Mww^*,
	\qquad
	w=(0,\beta,\alpha)^T,
	\]
	where $C_0$ is independent of $M$. Moreover,
	\[
	d:=\|w\|^2=\frac{a+b}{2},
	\qquad
	\langle e_1,C_0w\rangle=2\tau\alpha\beta\neq0.
	\]
	By Lemma \ref{lem:rank-one-fixed},
	\[
	\Tr P(A^{1/2}B_MA^{1/2})^s
	=\langle e_1,(C_0+Mww^*)^s e_1\rangle
	\longrightarrow +\infty
	\]
	as $M\to\infty$, because $s>2$. This diverges while
	\eqref{eq:monotone-left-bounded} is bounded. Thus
	\eqref{eq:monotone-failure} holds for all sufficiently large $M$.
	
	If one requires the weight to be strictly positive, fix such an $M$ and replace
	$f$ by $f_\delta=f+\delta$. The difference between the two sides of
	\eqref{eq:monotone-failure} depends continuously on $\delta$, and it is already
	negative at $\delta=0$. Hence the failure persists for all sufficiently small
	$\delta>0$.
\end{proof}

The preceding counterexample also disproves the two $s>2$ trace conjectures
formulated by Shi-Wei-Wang.

\begin{corollary}\label{cor:SWW-trace-conjectures-false}
	For every $s>2$, Conjectures 3.2 and 3.3 of Shi-Wei-Wang
	\cite{ShiWeiWang} fail already for $3\times3$ positive definite matrices.
\end{corollary}

\begin{proof}
	Conjecture 3.2 is exactly contradicted by
	Theorem~\ref{thm:monotone-counter}.
	
	We now disprove Conjecture 3.3. Let $A,B>0$ and $f\geq0$ be the counterexample
	from Theorem~\ref{thm:monotone-counter}, so that
	\[
	\Tr f(A)A^sB^s
	<
	\Tr f(A)(A^{1/2}BA^{1/2})^s .
	\]
	Set
	\[
	\widehat A=A^{-1},
	\qquad
	\widehat B=A^{1/2}BA^{1/2},
	\qquad
	g(x)=x^{-s}f(x^{-1}).
	\]
	Then
	$
	x^s g(x)=f(x^{-1})
	$
	is nonnegative and nonincreasing on the relevant spectral interval. Moreover,
	$
	\widehat A^{1/2}\widehat B\widehat A^{1/2}=B
	$
	and
	$
	g(\widehat A)=g(A^{-1})=A^s f(A).
	$
	Therefore
	\[
	\begin{aligned}
		\Tr g(\widehat A)
		(\widehat A^{1/2}\widehat B\widehat A^{1/2})^s
		&=
		\Tr A^sf(A)B^s
		=
		\Tr f(A)A^sB^s,\\
		\Tr g(\widehat A)\widehat A^s\widehat B^s
		&=
		\Tr A^sf(A)A^{-s}(A^{1/2}BA^{1/2})^s\\
		&=
		\Tr f(A)(A^{1/2}BA^{1/2})^s.
	\end{aligned}
	\]
	Thus the asserted inequality in Conjecture 3.3 would give the opposite of the
	strict counterexample above. Hence Conjecture 3.3 is false.
\end{proof}

The preceding counterexamples cannot be repaired by a constant depending only on
$s$ and on the spectral ratio of $A$.

\begin{corollary}\label{cor:no-kantorovich-monotone}
	Fix $s>2$ and $0<\kappa<1$. Then
	\[
	\sup
	\frac{
		\Tr f(A)(A^{1/2}BA^{1/2})^s
	}{
		\Tr f(A)A^sB^s
	}
	=+\infty,
	\]
	where the supremum is taken over all $A,B\in\mathbb P_3$ satisfying
	$
	\frac{\lambda_{\min}(A)}{\lambda_{\max}(A)}=\kappa,
	$
	and over all nonzero nonnegative nondecreasing functions $f$ on an interval
	containing $\spec(A)$ for which the denominator is positive. Consequently, no
	finite constant depending only on $s$ and on the spectral ratio
	$\lambda_{\min}(A)/\lambda_{\max}(A)$ can restore the monotone-weight inequality
	for all such weights.
\end{corollary}

\begin{proof}
	Choose $b=\kappa$ and any $a$ with $\kappa<a<1$. Use the matrices $A$ and
	$B_M$ and the weight $f$ from the proof of
	Theorem~\ref{thm:monotone-counter}. Then
	\[
	\lambda_{\max}(A)=1,
	\qquad
	\lambda_{\min}(A)=\kappa.
	\]
	Moreover,
	\[
	\Tr f(A)A^sB_M^s
	=
	\langle e_1,B_M^s e_1\rangle
	\]
	is independent of $M$, while Lemma~\ref{lem:rank-one-fixed} gives
	\[
	\Tr f(A)(A^{1/2}B_MA^{1/2})^s
	\geq cM^{s-2}(1+o(1))
	\]
	for some $c>0$. Hence the displayed ratio tends to $+\infty$ as
	$M\to\infty$.
\end{proof}

The same unboundedness persists if the weights are required to be strictly
positive, provided that no uniform positive lower bound is imposed. Indeed, the
weight is allowed to vary with \(M\), because the supremum in
Corollary~\ref{cor:no-kantorovich-monotone} is taken jointly over \(A,B\) and
over all admissible weights. In the counterexample above replace \(f\) by
$
f_M=f+M^{-3}.
$
Then \(f_M\) is strictly positive and nondecreasing. The additional identity-weight part contributes at
most $O(M^{s-3})$ to the denominator, whereas the projection part of the
numerator is of order $M^{s-2}$. Hence the ratio remains unbounded as
$M\to\infty$.

The next estimate is the scaled version of the preceding rank-one blow-up
mechanism. It will provide the sharp obstruction for the power-weight problem.

\begin{lemma}\label{lem:scaled-blowup}
	Let $w,z\in\bbC^2$ with $w\neq0$ and $z^*w\neq0$. Let $a_\eps\in\mathbb R$ and
	$S_\eps\in\overline{\mathbb P}_2$ be uniformly bounded as $\eps\downarrow0$.
	For $t>0$, set
	\[
	T_{\eps,t}
	=
	\begin{pmatrix}
		a_\eps&\sqrt\eps\, z^*\\
		\sqrt\eps\, z&S_\eps+tww^*
	\end{pmatrix},
	\]
	and assume that $T_{\eps,t}\geq0$ for the values of $\eps$ and $t$ under
	consideration. Put $d=\|w\|^2$. If $t=t_\eps\to\infty$ as $\eps\downarrow0$,
	then the largest eigenvalue $\mu_{\eps,t}$ of $T_{\eps,t}$ satisfies
	\[
	\mu_{\eps,t}=dt+O(1).
	\]
	Moreover, if $q_{\eps,t}$ is a corresponding unit eigenvector, then
	\begin{equation}\label{eq:scaled-component}
		|\langle e_1,q_{\eps,t}\rangle|^2
		=
		\frac{\eps |z^*w|^2}{d^3t^2}(1+o(1)).
	\end{equation}
	Consequently, for every $s>2$,
	\begin{equation}\label{eq:scaled-s-power}
		\langle e_1,T_{\eps,t}^s e_1\rangle
		\geq c\,\eps t^{s-2}(1+o(1))
	\end{equation}
	for some constant $c>0$ independent of $\eps$ and $t$.
\end{lemma}
\begin{proof}
	Let
	\[
	\widehat w=(0,w)^T,
	\qquad
	D_\eps=
	\begin{pmatrix}
		a_\eps&\sqrt\eps\, z^*\\
		\sqrt\eps\, z&S_\eps
	\end{pmatrix}.
	\]
	Then \(T_{\eps,t}=D_\eps+t\widehat w\widehat w^*\) and
	\(\|D_\eps\|\le C\) uniformly for small \(\eps\). The rank-one matrix
	\(\widehat w\widehat w^*\) has the simple nonzero eigenvalue
	\(d=\|w\|^2\). By the min-max principle, and since $\|D_\eps\|\leq C$, the largest eigenvalue
	belongs to $[dt-C,dt+C]$, while all the remaining eigenvalues belong to
	$[-C,C]$. Hence, for
	\(t\) sufficiently large, this is the unique eigenvalue separated from the
	remaining spectrum. Therefore, the spectral gap between the
	large eigenvalue and the rest of the spectrum is \(dt+O(1)\), uniformly in
	\(\eps\). In particular,
	\[
	\mu_{\eps,t}=dt+O(1).
	\]
	Let \(P\) be the orthogonal projection onto \(\mathbb C\widehat w\). Projecting
	the eigenvalue equation onto \(P^\perp\), we get
	\[
	\bigl(\mu_{\eps,t}I-P^\perp D_\eps P^\perp\bigr)P^\perp q_{\eps,t}
	=
	P^\perp D_\eps Pq_{\eps,t}.
	\]
	Since \(\mu_{\eps,t}=dt+O(1)\) and \(\|D_\eps\|\le C\), the operator on the
	left is invertible for large \(t\), and its inverse has norm \(O(t^{-1})\),
	uniformly in \(\eps\). Hence
	\[
	\|P^\perp q_{\eps,t}\|=O(t^{-1}).
	\]
	After choosing the phase of \(q_{\eps,t}\), we may therefore write
	\[
	q_{\eps,t}=(x_{\eps,t},y_{\eps,t}),
	\qquad
	y_{\eps,t}\to \frac{w}{\sqrt d}
	\]
	uniformly along every path with \(t=t_\eps\to\infty\).
	
	The first row of the eigenvalue equation gives
	\[
	(\mu_{\eps,t}-a_\eps)x_{\eps,t}
	=
	\sqrt\eps\, z^*y_{\eps,t}.
	\]
	Since \(z^*w\ne0\) and \(\mu_{\eps,t}=dt(1+o(1))\), we obtain
	\[
	x_{\eps,t}
	=
	\frac{\sqrt\eps\, z^*w}{d^{3/2}t}(1+o(1)),
	\]
	which proves \eqref{eq:scaled-component}. Because \(T_{\eps,t}\ge0\), the
	contribution of the largest eigenvalue to the spectral decomposition gives
	\[
	\langle e_1,T_{\eps,t}^s e_1\rangle
	\ge
	\mu_{\eps,t}^s|x_{\eps,t}|^2
	=
	\eps |z^*w|^2d^{s-3}t^{s-2}(1+o(1)).
	\]
	After decreasing the constant if necessary, this proves
	\eqref{eq:scaled-s-power} with a constant independent of \(\eps\) and \(t\).
\end{proof}

For $r,s\geq0$ and $A,B\geq0$, set
\begin{equation}\label{eq:Phi}
	\Phi_{r,s}(A,B)
	:=
	\Tr A^{r+s}B^s
	-
	\Tr A^r(A^{1/2}BA^{1/2})^s.
\end{equation}

The next theorem gives the necessity part of the sharp phase diagram: beyond the
hyperbola $r=s/(s-2)$, even the trace inequality fails.

\begin{theorem}\label{thm:power-counter}
	Let $s>2$ and assume that
	\begin{equation}\label{eq:threshold-condition}
		r>\frac{s}{s-2}.
	\end{equation}
	Then there exist $A_\eps,B_\eps\in\mathbb P_3$ such that
	\[
	\Phi_{r,s}(A_\eps,B_\eps)<0
	\]
	for all sufficiently small $\eps>0$. In fact,
	\[
	\Phi_{r,s}(A_\eps,B_\eps)\to -\infty
	\qquad (\eps\downarrow0).
	\]
\end{theorem}

\begin{proof}
	Choose $0<b<a<1$, $\tau>0$, $\eta>\tau^2$, and the same unitary $U$ as in the
	proof of Theorem \ref{thm:monotone-counter}. Since \eqref{eq:threshold-condition}
	is equivalent to
	\[
	\frac{s-1}{s-2}<\frac{r+1}{2},
	\]
	we can choose $\gamma$ such that
	\begin{equation}\label{eq:gamma}
		\frac{s-1}{s-2}<\gamma<\frac{r+1}{2}.
	\end{equation}
	Set
	\[
	A_\eps=\diag(1,\eps a,\eps b),
	\qquad
	M_\eps=\eps^{-\gamma},
	\]
	and
	\[
	B_\eps
	=U
	\begin{pmatrix}
		1&\tau&0\\
		\tau&\eta&0\\
		0&0&M_\eps
	\end{pmatrix}
	U^*.
	\]
	All matrices are positive definite.
	
	We estimate the first trace in the basis $\{e_1,u,v\}$. For every $t\geq0$,
	\[
	U^*A_\eps^tU=1\oplus \eps^t R_t,
	\]
	where $R_t$ is a fixed positive definite $2\times2$ matrix depending only on
	$a,b,t$. Also,
	\[
	U^*B_\eps^sU=K^s\oplus M_\eps^s,
	\qquad
	K=\begin{pmatrix}1&\tau\\ \tau&\eta\end{pmatrix}.
	\]
	Therefore
	\begin{equation}\label{eq:power-left}
		\Tr A_\eps^{r+s}B_\eps^s
		=O(1)+O(\eps^{r+s}M_\eps^s).
	\end{equation}
	
	Next set
	\[
	C_\eps:=U^*A_\eps^{1/2}B_\eps A_\eps^{1/2}U.
	\]
	With
	\[
	\alpha=\frac{\sqrt a+\sqrt b}{2},
	\qquad
	\beta=\frac{\sqrt a-\sqrt b}{2},
	\]
	one has
	\[
	U^*A_\eps^{1/2}U
	=
	\begin{pmatrix}
		1&0&0\\
		0&\sqrt\eps\,\alpha&\sqrt\eps\,\beta\\
		0&\sqrt\eps\,\beta&\sqrt\eps\,\alpha
	\end{pmatrix}.
	\]
	Consequently
	\[
	C_\eps
	=
	\begin{pmatrix}
		1 & \sqrt\eps\, z^*\\
		\sqrt\eps\, z & S_\eps+t_\eps ww^*
	\end{pmatrix},
	\]
	where \(S_\eps\ge0\) is uniformly bounded and
	\[
	t_\eps=\eps M_\eps=\eps^{1-\gamma},
	\qquad
	w=(\beta,\alpha)^T,
	\qquad
	z=\tau(\alpha,\beta)^T.
	\]
	Thus \(C_\eps\) has the form required in
	Lemma~\ref{lem:scaled-blowup}. Since \(a\neq b\) and \(\tau>0\),
	\[
	z^*w=2\tau\alpha\beta\neq0.
	\]
	The left inequality in \eqref{eq:gamma} implies $\gamma>1$, hence
	$t_\eps\to\infty$. Lemma \ref{lem:scaled-blowup} gives
	\begin{equation}\label{eq:power-right-lower}
		\langle e_1,C_\eps^s e_1\rangle
		\geq c\,\eps t_\eps^{s-2}(1+o(1))
	\end{equation}
	for some $c>0$.
	Since $U^*A_\eps^rU=1\oplus\eps^rR_r\geq e_1e_1^*$, we have
	\begin{equation}\label{eq:weighted-right-lower}
		\Tr A_\eps^r(A_\eps^{1/2}B_\eps A_\eps^{1/2})^s
		\geq c\,\eps t_\eps^{s-2}(1+o(1)).
	\end{equation}
	Combining \eqref{eq:power-left} and \eqref{eq:weighted-right-lower},
	\begin{equation}\label{eq:power-Phi-upper}
		\Phi_{r,s}(A_\eps,B_\eps)
		\leq
		O(1)+O(\eps^{r+s}M_\eps^s)
		-c\,\eps t_\eps^{s-2}(1+o(1)).
	\end{equation}
	Now
	\[
	\eps t_\eps^{s-2}
	=\eps^{s-1-\gamma(s-2)}\to\infty
	\]
	by the left inequality in \eqref{eq:gamma}. Moreover,
	\[
	\frac{\eps^{r+s}M_\eps^s}{\eps t_\eps^{s-2}}
	=\eps^{r+1-2\gamma}\to0
	\]
	by the right inequality in \eqref{eq:gamma}. Thus the negative term in
	\eqref{eq:power-Phi-upper} dominates, and
	$\Phi_{r,s}(A_\eps,B_\eps)\to-\infty$.
\end{proof}

Theorem~\ref{thm:power-counter} rules out the Liu-Cheng
log-majorization conjecture beyond the sharp hyperbola. Indeed, if
\eqref{eq:LC-conj2} held with $p=1$ and $q=s>2$, then, since
log-majorization implies the corresponding trace inequality, one would have
\[
\Tr A^{r+s}B^s\geq
\Tr A^r(A^{1/2}BA^{1/2})^s
\]
for all $A,B\geq0$. This contradicts Theorem~\ref{thm:power-counter} whenever
$r>s/(s-2)$.

\begin{corollary}\label{cor:LC2-false}
	The conjectured all-$r$ log-majorization \eqref{eq:LC-conj2} does not hold in
	the full stated range. More precisely,
	for every $q>2$ and every $r>q/(q-2)$, it fails with $p=1$ already for
	$3\times3$ positive definite matrices.
\end{corollary}

\section{The exact power-weight log-majorization range}\label{sec:tanahashi}

The counterexamples above give the necessity part of the sharp phase diagram.
We now prove the matching sufficiency. The range
$
1\leq s\leq2
$
in the normalized variables, equivalently
$
p\leq q\leq2p
$
in the original variables, was proved by Shi-Wei-Wang
\cite{ShiWeiWang}. It remains to prove that, for $s>2$, the only obstruction is
the hyperbola
$
r=\frac{s}{s-2}.
$
The main tool is the following notation-converted form of Tanahashi's
negative-power Furuta inequality \cite{Tanahashi}. We state it only in the form
needed below; the subsequent lemma uses its diagonal specialization \(p=q\).

\begin{theorem}[Tanahashi]\label{thm:Tanahashi}
	Let \(X,Y>0\) with \(Y\leq X\). Let \(0<p\leq1\), \(0<q\leq1\), and
	let \(\rho<0\) satisfy
	$
	-1\leq 2\rho<0.
	$
	Then
	\begin{equation}\label{eq:Tanahashi}
		(X^\rho Y^pX^\rho)^{1/q}
		\leq X^{(p+2\rho)/q}
	\end{equation}
	provided that
	\begin{equation}\label{eq:T1}
		-2\rho(1-q)\leq p\leq q-2\rho(1-q),
	\end{equation}
	and, in addition, either
	\[
	\frac12\leq q\leq1,
	\]
	or
	\[
	0<q<\frac12
	\]
	and
	\begin{equation}\label{eq:T2}
		\frac{-2\rho(1-q)-q}{1-2q}\leq p\leq
		\frac{-2\rho(1-q)}{1-2q}.
	\end{equation}
\end{theorem}

We shall use only the diagonal specialization $p=q$ in Tanahashi's notation.
The next lemma gives the form needed below; the endpoint $r=s/(s-2)$ in the main
log-majorization theorem will be obtained later by a limiting argument.

\begin{lemma}\label{lem:diag-negative-furuta}
	Let $0<Y\leq X$, let $0<\alpha\leq1$, and let $0\leq\beta<1$. Assume either
	$
	\alpha\geq\frac12,
	$
	or
	$
	0<\alpha<\frac12
	\quad\text{and}\quad
	\beta<\frac{\alpha}{1-\alpha}.
	$
	Then
	\begin{equation}\label{eq:diag-negative-furuta}
		\bigl(X^{-\alpha(1-\beta)/2}Y^\alpha
		X^{-\alpha(1-\beta)/2}\bigr)^{1/\alpha}
		\leq X^\beta.
	\end{equation}
\end{lemma}

\begin{proof}
	Apply Theorem \ref{thm:Tanahashi} with
	\[
	p=q=\alpha,
	\qquad
	\rho=-\frac{\alpha(1-\beta)}2.
	\]
	Then $-1\leq2\rho<0$. Condition \eqref{eq:T1} becomes
	\[
	\alpha(1-\beta)(1-\alpha)\leq\alpha
	\leq \alpha+\alpha(1-\beta)(1-\alpha),
	\]
	which is automatic. If $\alpha\geq1/2$, we are done. If $0<\alpha<1/2$, the
	left inequality in \eqref{eq:T2} is also automatic, while the right inequality in
	\eqref{eq:T2} is
	\[
	\alpha\leq
	\frac{\alpha(1-\beta)(1-\alpha)}{1-2\alpha},
	\]
	which is equivalent to
	$
	\beta\leq\frac{\alpha}{1-\alpha}.
	$
	In the lemma we impose the slightly stronger open condition
	$
	\beta<\frac{\alpha}{1-\alpha},
	$
	because this is the only range needed before the endpoint limiting argument.
	The endpoint case is not taken from Tanahashi at this point; it is obtained later
	by the continuity argument in Theorem~\ref{thm:sharp-power-log}. Thus Tanahashi's
	theorem yields \eqref{eq:diag-negative-furuta}.
\end{proof}

We now state the exact power-weight reverse log-majorization range in the
normalized variables $p=1$ and $q=s$.

\begin{theorem}\label{thm:sharp-power-log}
	Let $s\geq1$ and $r\geq0$. The log-majorization
	\begin{equation}\label{eq:sharp-power-log}
		A^{r+s}B^s\succ_{\log}
		A^r(A^{1/2}BA^{1/2})^s
	\end{equation}
	holds for all positive semidefinite matrices $A,B$ in all finite dimensions if
	and only if either
	\begin{equation}\label{eq:range-power-1}
		1\leq s\leq2,
	\end{equation}
	or
	\begin{equation}\label{eq:range-power-2}
		s>2
		\quad\text{and}\quad
		0\leq r\leq\frac{s}{s-2}.
	\end{equation}
	Consequently, the trace inequality
	\begin{equation}\label{eq:sharp-power-trace}
		\Tr A^{r+s}B^s\geq
		\Tr A^r(A^{1/2}BA^{1/2})^s
	\end{equation}
	holds universally in exactly the same parameter range.
\end{theorem}

The known and new parts of Theorem~\ref{thm:sharp-power-log} should be
distinguished. The range $1\leq s\leq2$ is due to Shi-Wei-Wang
\cite{ShiWeiWang}. For $s>2$, the subrange
$
0\leq r\leq1
$
is already contained in the known reverse GBLP range
$
0\leq r\leq p\leq q
$
after setting $p=1$ and $q=s$. Hence the new affirmative part for $s>2$ is
$
1<r\leq \frac{s}{s-2},
$
and the matching new obstruction is the failure for
$
r>\frac{s}{s-2}.
$

\begin{proof}
	The case $1\leq s\leq2$ is the all-$r$ range proved by
	Shi-Wei-Wang \cite[Theorem~2.1]{ShiWeiWang}, after setting $p=1$ and $q=s$.  The trace consequence follows from log-majorization. Therefore the
	new sufficiency to prove here is the range
	\[
	s>2,
	\qquad
	0\le r\le \frac{s}{s-2}.
	\]
	
	We first prove this positive part for \(A,B>0\). Put
	$
	C=A^{1/2}BA^{1/2},
	$
	and introduce the positive semidefinite Hermitian representatives
	\[
	X_{A,B}:=A^{(r+s)/2}B^sA^{(r+s)/2},
	\qquad
	Y_{A,B}:=A^{r/2}C^sA^{r/2}.
	\]
	These representatives have the eigenvalues of the products appearing in
	\eqref{eq:sharp-power-log}. Indeed,
	\[
	A^{(r+s)/2}B^sA^{(r+s)/2}
	\sim B^sA^{r+s}\sim A^{r+s}B^s,
	\]
	and similarly
	\[
	A^{r/2}C^sA^{r/2}
	\sim C^sA^r\sim A^rC^s.
	\]
	Therefore it is enough to prove
	$
	Y_{A,B}\prec_{\log} X_{A,B}.
	$
	We first prove the largest-eigenvalue comparison
	\begin{equation}\label{eq:lmax-comparison}
		\lambda_{\max}(Y_{A,B})\leq \lambda_{\max}(X_{A,B}).
	\end{equation}
	By homogeneity in $B$, it suffices to prove that
	\[
	X_{A,B}\leq I
	\quad\Longrightarrow\quad
	Y_{A,B}\leq I.
	\]
	Indeed, if $\lambda_{\max}(X_{A,B})=L$, replace $B$ by $L^{-1/s}B$.
	
	Assume therefore that
	$
	A^{(r+s)/2}B^sA^{(r+s)/2}\leq I.
	$
	Equivalently,
	\begin{equation}\label{eq:B-upper}
		B^s\leq A^{-(r+s)}.
	\end{equation}
	Set
	\[
	X=A^{-(r+s)},
	\qquad
	Y=B^s,
	\qquad
	\alpha=\frac1s,
	\qquad
	\beta=\frac{r}{r+s}.
	\]
	Then $0<Y\leq X$ by \eqref{eq:B-upper}. Moreover,
	\[
	X^{-\alpha(1-\beta)/2}=A^{1/2},
	\qquad
	Y^\alpha=B,
	\qquad
	X^\beta=A^{-r}.
	\]
	
	Assume first that
	\[
	s>2,
	\qquad
	0\le r<\frac{s}{s-2}.
	\]
	Then $\alpha<1/2$ and
	\[
	\beta<\frac{\alpha}{1-\alpha}
	\quad\Longleftrightarrow\quad
	r<\frac{s}{s-2}.
	\]
	By Lemma~\ref{lem:diag-negative-furuta},
	\[
	C^s=(A^{1/2}BA^{1/2})^s\leq A^{-r}.
	\]
	Therefore
	\[
	Y_{A,B}=A^{r/2}C^sA^{r/2}\leq I,
	\]
	which proves \eqref{eq:lmax-comparison} in the open range.
	
	It remains to treat the endpoint $s>2$ and $r=s/(s-2)$. Choose
	$r_j\uparrow r$ with $r_j<s/(s-2)$. The preceding argument gives the
	largest-eigenvalue comparison for each $r_j$. Since
	\[
	A^{(r_j+s)/2}B^sA^{(r_j+s)/2}
	\to
	A^{(r+s)/2}B^sA^{(r+s)/2},
	\qquad
	A^{r_j/2}C^sA^{r_j/2}\to A^{r/2}C^sA^{r/2}
	\]
	in norm, the comparison passes to the limit. Hence
	\eqref{eq:lmax-comparison} also holds at the endpoint.
	
	To pass from the largest eigenvalue to log-majorization, we use the standard
	antisymmetric tensor power argument; see, for example, \cite{AndoHiai}. For
	\(1\le k\le n\), apply the already proved largest-eigenvalue comparison to the
	positive matrices \(\wedge^k A\) and \(\wedge^k B\). Since
	\[
	\wedge^k(A^{1/2}BA^{1/2})
	=
	(\wedge^k A)^{1/2}(\wedge^k B)(\wedge^k A)^{1/2},
	\]
	the corresponding Hermitian representatives are precisely
	\[
	\wedge^k Y_{A,B}
	\quad\text{and}\quad
	\wedge^k X_{A,B}.
	\]
	Therefore
	\[
	\lambda_{\max}(\wedge^k Y_{A,B})
	\le
	\lambda_{\max}(\wedge^k X_{A,B}).
	\]
	Using
	\[
	\lambda_{\max}(\wedge^k Y_{A,B})
	=
	\prod_{j=1}^k\lambda_j(Y_{A,B}),
	\qquad
	\lambda_{\max}(\wedge^k X_{A,B})
	=
	\prod_{j=1}^k\lambda_j(X_{A,B}),
	\]
	we obtain, for every \(1\le k\le n\),
	\[
	\prod_{j=1}^k\lambda_j(Y_{A,B})
	\leq
	\prod_{j=1}^k\lambda_j(X_{A,B}).
	\]
	For $k=n$, the two products are equal because
	\[
	\det Y_{A,B}
	=\det(A)^r\det(C)^s
	=\det(A)^{r+s}\det(B)^s
	=\det X_{A,B}.
	\]
	Hence $Y_{A,B}\prec_{\log}X_{A,B}$.
	
	The semidefinite case follows by continuity. For $\delta>0$, apply the positive
	definite case to
	\[
	A_\delta=A+\delta I,
	\qquad
	B_\delta=B+\delta I.
	\]
	The Hermitian representatives
	\[
	A_\delta^{(r+s)/2}B_\delta^sA_\delta^{(r+s)/2},
	\qquad
	A_\delta^{r/2}
	(A_\delta^{1/2}B_\delta A_\delta^{1/2})^s
	A_\delta^{r/2}
	\]
	converge in norm to the corresponding representatives for $A,B$. Eigenvalues
	and exterior-power eigenvalue products are continuous under norm convergence,
	so the log-majorization inequalities pass to the limit as $\delta\downarrow0$.
	
	For necessity, suppose $s>2$ and $r>s/(s-2)$. Theorem~\ref{thm:power-counter}
	gives positive definite $A,B$ for which the trace inequality
	\eqref{eq:sharp-power-trace} fails. Therefore the log-majorization
	\eqref{eq:sharp-power-log} cannot hold universally. This proves both the
	necessity for log-majorization and the necessity for the trace inequality.
\end{proof}

The endpoint estimate also gives a stability bound outside the sharp
power-weight range, with a constant depending only on the spectral ratio of
$A$.

\begin{corollary}\label{cor:power-kantorovich}
	Fix $s>2$, set $R_s=s/(s-2)$, and take $A\in\Pn$, $B\in\Psn$. Put
	$
	\kappa(A)=\frac{\lambda_{\min}(A)}{\lambda_{\max}(A)}.
	$
	Then, for every $r\geq0$,
	\begin{equation}\label{eq:power-kantorovich}
		\Tr A^r(A^{1/2}BA^{1/2})^s
		\leq
		\kappa(A)^{-(r-R_s)_+}
		\Tr A^{r+s}B^s .
	\end{equation}
	In particular, the constant is $1$ throughout the sharp positive range
	$0\leq r\leq R_s$.
\end{corollary}

\begin{proof}
	By homogeneity, normalize $\lambda_{\max}(A)=1$. Then
	\[
	\kappa(A)I\leq A\leq I.
	\]
	If $0\leq r\leq R_s$, the assertion is precisely the trace consequence of
	Theorem~\ref{thm:sharp-power-log}.
	
	Assume now that $r>R_s$. Since $A\leq I$, we have
	$
	A^r\leq A^{R_s}.
	$
	Therefore
	\[
	\Tr A^r(A^{1/2}BA^{1/2})^s
	\leq
	\Tr A^{R_s}(A^{1/2}BA^{1/2})^s.
	\]
	By Theorem~\ref{thm:sharp-power-log} at the endpoint $R_s$,
	\[
	\Tr A^{R_s}(A^{1/2}BA^{1/2})^s
	\leq
	\Tr A^{R_s+s}B^s.
	\]
	Finally, since $A\geq\kappa(A)I$ and $r>R_s$,
	\[
	A^{R_s+s}
	=
	A^{r+s}A^{R_s-r}
	\leq
	\kappa(A)^{-(r-R_s)}A^{r+s}.
	\]
	Taking the trace against $B^s\geq0$ gives
	\[
	\Tr A^{R_s+s}B^s
	\leq
	\kappa(A)^{-(r-R_s)}
	\Tr A^{r+s}B^s.
	\]
	Combining the preceding inequalities proves \eqref{eq:power-kantorovich}.
\end{proof}

The same endpoint estimate gives a weighted stability bound. The constant below
measures how far the weight is from the endpoint power $x^{R_s}$ on the
spectrum of $A$.

\begin{proposition}\label{prop:weighted-stability}
	Let $s>2$ and put
	$
	R_s=\frac{s}{s-2}.
	$
	Let $A\in\Pn$, $B\in\Psn$, and let $f$ be strictly positive on $\spec(A)$. Define
	\[
	\Omega_{s,A}(f)
	=
	\frac{
		\max_{\lambda\in\spec(A)} f(\lambda)\lambda^{-R_s}
	}{
		\min_{\lambda\in\spec(A)} f(\lambda)\lambda^{-R_s}
	}.
	\]
	Then
	\begin{equation}\label{eq:weighted-stability}
		\Tr f(A)(A^{1/2}BA^{1/2})^s
		\leq
		\Omega_{s,A}(f)
		\Tr f(A)A^sB^s.
	\end{equation}
\end{proposition}

\begin{proof}
	Set
	\[
	C=A^{1/2}BA^{1/2},
	\qquad
	g(A)=f(A)A^{-R_s}.
	\]
	Then
	\[
	g_{\min}I\leq g(A)\leq g_{\max}I,
	\]
	where
	\[
	g_{\min}
	=
	\min_{\lambda\in\spec(A)} f(\lambda)\lambda^{-R_s},
	\qquad
	g_{\max}
	=
	\max_{\lambda\in\spec(A)} f(\lambda)\lambda^{-R_s}.
	\]
	Since $g(A)$ commutes with powers of $A$, we have
	\[
	f(A)=A^{R_s/2}g(A)A^{R_s/2}.
	\]
	Hence, using $C^s\ge0$ and trace monotonicity,
	\[
	\begin{aligned}
		\Tr f(A)C^s
		&=
		\Tr A^{R_s/2}g(A)A^{R_s/2}C^s\\
		&\leq
		g_{\max}\Tr A^{R_s}C^s.
	\end{aligned}
	\]
	By Theorem~\ref{thm:sharp-power-log} at the endpoint $R_s$,
	\[
	\Tr A^{R_s}C^s
	\leq
	\Tr A^{R_s+s}B^s.
	\]
	On the other hand,
	\[
	f(A)A^s
	=
	A^{(R_s+s)/2}g(A)A^{(R_s+s)/2},
	\]
	and therefore, using $B^s\ge0$,
	\[
	\begin{aligned}
		\Tr f(A)A^sB^s
		&=
		\Tr A^{(R_s+s)/2}g(A)A^{(R_s+s)/2}B^s\\
		&\geq
		g_{\min}\Tr A^{R_s+s}B^s.
	\end{aligned}
	\]
	Combining the preceding estimates gives
	\[
	\Tr f(A)(A^{1/2}BA^{1/2})^s
	\leq
	\frac{g_{\max}}{g_{\min}}
	\Tr f(A)A^sB^s,
	\]
	which is \eqref{eq:weighted-stability}.
\end{proof}

\begin{corollary}\label{cor:phase}
	For $s\geq1$ define
	\[
	\cR_s:=\left\{r\geq0:
	\Tr A^{r+s}B^s\geq
	\Tr A^r(A^{1/2}BA^{1/2})^s
	\text{ for all } A,B\geq0\right\}.
	\]
	Then
	\begin{equation}\label{eq:phase-exact}
		\cR_s=[0,\infty),\qquad 1\leq s\leq2,
	\end{equation}
	and
	\begin{equation}\label{eq:phase-exact-2}
		\cR_s=\left[0,\frac{s}{s-2}\right],
		\qquad s>2.
	\end{equation}
	For $0<s<1$, the inequality has the opposite sign for every $r\geq0$, by
	\eqref{eq:LC-sublinear}. At $s=1$ the two sides are equal by cyclicity of the
	trace.
\end{corollary}
\section{The monotone-weight phase transition and minimal dimension}
\label{sec:monotone}

We shall use the following theorem of Shi-Wei-Wang \cite[Theorem~2.4]{ShiWeiWang}.

\begin{theorem}\label{thm:SWW-monotone-positive}
	Let $A,B\in\Psn$, and let $f$ be nonnegative and nondecreasing on an interval
	containing $\spec(A)$. Then
	\[
	\Tr f(A)A^sB^s
	\geq
	\Tr f(A)(A^{1/2}BA^{1/2})^s,
	\qquad 1\leq s\leq2.
	\]
\end{theorem}

Combining the positive range of Shi-Wei-Wang with
Theorem~\ref{thm:monotone-counter} gives the sharp correction of the
monotone-weight conjecture: the range $1\leq s\leq2$ holds universally, while
every $s>2$ fails already in dimension $3$. Together with Liu-Cheng's
sublinear theorem, this yields the following complete monotone-weight picture.

\begin{corollary}\label{cor:corrected-conj1}
	For fixed $s\geq1$, the inequality
	\[
	\Tr f(A)A^sB^s
	\geq
	\Tr f(A)(A^{1/2}BA^{1/2})^s
	\]
	holds for all $A,B\in\Psn$ in all finite dimensions and all nonnegative
	nondecreasing weights $f$ on an interval containing $\spec(A)$ if and only if
	\[
	1\leq s\leq2.
	\]
	Moreover, the full monotone-weight picture is
	\[
	\begin{array}{ll}
		0\leq s\leq1:&
		\Tr f(A)A^sB^s\leq \Tr f(A)(A^{1/2}BA^{1/2})^s,\\[1mm]
		1\leq s\leq2:&
		\Tr f(A)A^sB^s\geq \Tr f(A)(A^{1/2}BA^{1/2})^s,\\[1mm]
		s>2:&
		\text{the reverse inequality fails in general already in dimension }3.
	\end{array}
	\]
\end{corollary}

\subsection{Minimal dimension of failure}

The counterexamples constructed above are three-dimensional. We now show that
this is optimal: in dimension $2$ the monotone-weight reverse trace inequality
holds for every $s\geq1$.
The key point is the following two-dimensional endpoint comparison.

\begin{lemma}\label{lem:two-by-two-endpoint}
	Let $0\leq a\leq1$, set
	\[
	D_a=\begin{pmatrix}1&0\\0&\sqrt a\end{pmatrix},
	\]
	and let $B\in\overline{\mathbb P}_2$. Then, for every $s\geq1$,
	\begin{equation}\label{eq:two-by-two-endpoint}
		\langle e_1,B^s e_1\rangle
		\geq
		\langle e_1,(D_aBD_a)^s e_1\rangle .
	\end{equation}
\end{lemma}

\begin{proof}
	If $B_{11}=0$, then positivity of $B$ gives $B_{12}=0$, and both sides of
	\eqref{eq:two-by-two-endpoint} are zero. Assume $B_{11}>0$. By homogeneity we
	may normalize $B_{11}=1$. After conjugating by a diagonal unitary which fixes
	$e_1$ and commutes with $D_a$, we may also assume that $B_{12}\geq0$. Hence
	\[
	B=
	\begin{pmatrix}
		1&\theta\sqrt u\\
		\theta\sqrt u&u
	\end{pmatrix}
	=:M_u,
	\qquad
	u\geq0,\quad 0\leq\theta\leq1.
	\]
	Then
	$
	D_aBD_a=M_{au}.
	$
	It is therefore enough to prove that
	$
	\psi_s(u):=\langle e_1,M_u^s e_1\rangle
	$
	is nondecreasing on $[0,\infty)$.
	The cases $\theta=0$ and $\theta=1$ are immediate. Indeed, if $\theta=0$, then
	$M_u=\operatorname{diag}(1,u)$ and $\psi_s(u)=1$; if $\theta=1$, then
	$M_u$ has rank one and
	$
	\psi_s(u)=(1+u)^{s-1},
	$
	which is nondecreasing for $s\geq1$.
	
	Assume now that $0<\theta<1$ and $u>0$. Let
	\[
	\Delta=\sqrt{(1-u)^2+4\theta^2u}
	\]
	and let
	\[
	\mu_\pm=\frac{1+u\pm\Delta}{2}
	\]
	be the eigenvalues of $M_u$. A direct spectral calculation gives
	\[
	\psi_s(u)
	=
	\frac{\mu_+^s+\mu_-^s}{2}
	+
	\frac{1-u}{2\Delta}(\mu_+^s-\mu_-^s).
	\]
	Differentiating this expression yields
	\[
	\psi_s'(u)
	=
	\frac{\theta^2(1-\theta^2)u}{\mu_+\mu_-\Delta^3}
	\left[
	(s\Delta+1+u)\mu_-^s+
	(s\Delta-1-u)\mu_+^s
	\right].
	\]
	Thus it remains to prove that the bracket is nonnegative. Put
	$
	r=\frac{\mu_-}{\mu_+}\in(0,1).
	$
	Since
	\[
	\frac{\Delta}{1+u}
	=
	\frac{\mu_+-\mu_-}{\mu_++\mu_-}
	=
	\frac{1-r}{1+r},
	\]
	the bracket equals
	\[
	\mu_+^s(1+u)
	\left[
	\left(1+s\frac{1-r}{1+r}\right)r^s
	+
	s\frac{1-r}{1+r}
	-1
	\right].
	\]
	The expression in square brackets is nonnegative precisely when
	\[
	\frac{1-r^s}{1+r^s}
	\leq
	s\frac{1-r}{1+r}.
	\]
	Writing $r=e^{-2x}$, $x\geq0$, this becomes
	\[
	\tanh(sx)\leq s\tanh x.
	\]
	This inequality follows because
	\[
	h(x)=s\tanh x-\tanh(sx)
	\]
	satisfies $h(0)=0$ and
	\[
	h'(x)
	=
	s\bigl(\operatorname{sech}^2x-\operatorname{sech}^2(sx)\bigr)
	\geq0
	\qquad (x\geq0,\ s\geq1).
	\]
	Hence $\psi_s'(u)\geq0$, and therefore
	\[
	\psi_s(u)\geq\psi_s(au),
	\qquad 0\leq a\leq1.
	\]
	This proves \eqref{eq:two-by-two-endpoint}.
\end{proof}

The endpoint comparison implies the full two-dimensional monotone-weight
inequality.

\begin{theorem}\label{thm:two-by-two-all-s}
	Let $A,B\in\overline{\mathbb P}_2$, let $f$ be a nonnegative nondecreasing
	function on an interval containing $\spec(A)$, and let $s\geq1$. Then
	\begin{equation}\label{eq:two-by-two-all-s}
		\Tr f(A)A^sB^s
		\geq
		\Tr f(A)(A^{1/2}BA^{1/2})^s .
	\end{equation}
\end{theorem}

\begin{proof}
	If $A$ is a scalar matrix, then equality holds. Otherwise, diagonalize $A$ as
	\[
	A=\lambda_1P_1+\lambda_2P_2,
	\qquad
	\lambda_1>\lambda_2\geq0.
	\]
	Write
	\[
	f(A)=f(\lambda_2)I+
	\bigl(f(\lambda_1)-f(\lambda_2)\bigr)P_1.
	\]
	Both coefficients are nonnegative.
	
	The identity part is handled by the Araki-Lieb-Thirring inequality:
	\[
	\Tr A^sB^s
	\geq
	\Tr(A^{1/2}BA^{1/2})^s .
	\]
	It remains to compare the $P_1$ part. If $\lambda_1=0$, then $A=0$ and there is
	nothing to prove. Otherwise put
	\[
	a=\frac{\lambda_2}{\lambda_1},
	\qquad
	D_a=\begin{pmatrix}1&0\\0&\sqrt a\end{pmatrix}
	\]
	in the eigenbasis of $A$. Then
	\[
	A=\lambda_1D_a^2,
	\qquad
	A^{1/2}BA^{1/2}
	=
	\lambda_1D_aBD_a.
	\]
	Lemma~\ref{lem:two-by-two-endpoint} gives
	\[
	\Tr P_1A^sB^s
	=
	\lambda_1^s\langle e_1,B^s e_1\rangle
	\geq
	\lambda_1^s\langle e_1,(D_aBD_a)^s e_1\rangle
	=
	\Tr P_1(A^{1/2}BA^{1/2})^s .
	\]
	Multiplying the identity comparison and the $P_1$ comparison by the two
	nonnegative coefficients in the decomposition of $f(A)$ proves
	\eqref{eq:two-by-two-all-s}.
\end{proof}

Theorem~\ref{thm:two-by-two-all-s} and Theorem~\ref{thm:monotone-counter}
show that dimension $3$ is exactly the first dimension in which the
monotone-weight reverse inequality can fail.

\begin{corollary}\label{cor:minimal-dimension}
	For $s>2$, dimension $3$ is the minimal dimension in which the monotone-weight
	reverse trace inequality can fail. More precisely, the inequality
	\[
	\Tr f(A)A^sB^s
	\geq
	\Tr f(A)(A^{1/2}BA^{1/2})^s
	\]
	holds for all $A,B\in\overline{\mathbb P}_2$ and all nonnegative
	nondecreasing weights $f$ on an interval containing $\spec(A)$, while for every
	$s>2$ there exist $A,B\in\mathbb P_3$ and such a weight $f$ for which the
	inequality is strictly reversed.
\end{corollary}

\section{Returning to the original $(p,q,r)$ variables}\label{sec:corrected}

We now translate Theorem~\ref{thm:sharp-power-log} back to the original
parameters $0<p\leq q$. This gives the sharp boundary for the reverse Liu--Cheng log-majorization
problem in the parameter range considered here.

\begin{theorem}\label{thm:corrected-GBLP}
	Let $0<p\leq q$ and $r\geq0$. The following are equivalent:
	\begin{enumerate}
		\item For every $n\geq1$ and all $A,B\in\overline{\mathbb P}_n$, one has
		\begin{equation}\label{eq:corrected-GBLP}
			A^{r+q}B^q\succ_{\log}
			A^r\bigl(A^{p/2}B^pA^{p/2}\bigr)^{q/p}.
		\end{equation}
		
		\item Either
		\begin{equation}\label{eq:corrected-range-1}
			p\leq q\leq2p,
		\end{equation}
		or
		\begin{equation}\label{eq:corrected-range-2}
			q>2p
			\quad\text{and}\quad
			0\leq r\leq\frac{pq}{q-2p}.
		\end{equation}
	\end{enumerate}
	Moreover, if $q>2p$ and $r>pq/(q-2p)$, then even the corresponding trace
	inequality fails for some $A,B\in\mathbb P_3$.
\end{theorem}

Before the proof, let us separate the known and new parts of
Theorem~\ref{thm:corrected-GBLP}. The range
$
p\leq q\leq2p,\qquad r\geq0,
$
is due to Shi-Wei-Wang \cite{ShiWeiWang}. In the remaining range $q>2p$, the
subrange
$
0\leq r\leq p
$
is already contained in the known reverse GBLP range; see, for example,
\cite{MatsumotoNakamotoFujii,Furuta2012,LemosSoares,ShiWeiWang}. Therefore the
new affirmative part of Theorem~\ref{thm:corrected-GBLP} is
$
p<r\leq \frac{pq}{q-2p},
$
and the matching new negative part is the failure for
$
r>\frac{pq}{q-2p}.
$

\begin{proof}
	Set
	\[
	\widetilde A=A^p,
	\qquad
	\widetilde B=B^p,
	\qquad
	s=\frac qp,
	\qquad
	\widetilde r=\frac rp.
	\]
	Then \eqref{eq:corrected-GBLP} is exactly
	\[
	\widetilde A^{\widetilde r+s}\widetilde B^s
	\succ_{\log}
	\widetilde A^{\widetilde r}
	(\widetilde A^{1/2}\widetilde B\widetilde A^{1/2})^s.
	\]
	Theorem \ref{thm:sharp-power-log} gives the positive range. Namely, if
	$1\leq s\leq2$, i.e. $p\leq q\leq2p$, all $\widetilde r\geq0$ are allowed. If
	$s>2$, i.e. $q>2p$, the condition is
	$
	\widetilde r\leq\frac{s}{s-2},
	$
	which is equivalent to
	$
	r\leq \frac{pq}{q-2p}.
	$
	This proves sufficiency.
	
	Conversely, if $q>2p$ and $r>pq/(q-2p)$, then
	$\widetilde r>s/(s-2)$. Theorem \ref{thm:power-counter} gives positive definite
	$\widetilde A,
	\widetilde B$ for which the trace inequality fails. Taking
	$A=\widetilde A^{1/p}$ and $B=\widetilde B^{1/p}$ gives a counterexample to the
	trace inequality corresponding to \eqref{eq:corrected-GBLP}. Hence
	\eqref{eq:corrected-GBLP} cannot hold universally.
\end{proof}

Applying Theorem~\ref{thm:corrected-GBLP} with $r=t$ gives the sharp form of
Shi-Wei-Wang's Conjecture 3.1.

\begin{corollary}\label{cor:SWW-conj31}
	Let $p>0$, $q>2p$, and $t\geq0$. Then
	\[
	A^{t+q}B^q
	\succ_{\log}
	A^t\bigl(A^{p/2}B^pA^{p/2}\bigr)^{q/p}
	\]
	holds for every $n\geq1$ and all $A,B\in\overline{\mathbb P}_n$ if and only if
	$
	0\leq t\leq \frac{pq}{q-2p}.
	$
	
	In particular, Conjecture 3.1 of Shi-Wei-Wang \cite{ShiWeiWang}, which asks
	for all $t\geq0$ when $q>2p$, is false as stated. If
	$
	t>\frac{pq}{q-2p},
	$
	then even the corresponding trace inequality fails for some
	$A,B\in\mathbb P_3$.
\end{corollary}

This proves the sharp range in the original $(p,q,r)$ variables and completes
the translation of the normalized theorem back to the reverse Liu--Cheng
parameter range.

\section{Conclusion}

We have identified two distinct sharpness mechanisms for reverse Araki-type
inequalities. For arbitrary monotone weights, the dimension-free reverse trace
inequality has the sharp transition
\[
1\leq s\leq2.
\]
For every $s>2$ it fails already in dimension $3$, while dimension $2$ remains
exceptional: the same monotone-weight reverse inequality holds there for every
$s\geq1$.

For power weights, the positive range extends beyond $s=2$ but only up to a
hyperbolic boundary. In normalized variables the sharp cutoff is
\[
r=\frac{s}{s-2},
\qquad s>2,
\]
and in the original variables it becomes
\[
r=\frac{pq}{q-2p},
\qquad q>2p.
\]
Thus the known all-$r$ range $p\leq q\leq2p$ is complemented, in the remaining
range $q>2p$, by the exact condition
\[
0\leq r\leq \frac{pq}{q-2p}.
\]
The counterexamples beyond this boundary are matching and already occur for
$3\times3$ positive definite matrices.




\end{document}